\documentclass{amsart}
\usepackage{graphicx}
\usepackage{latexsym}
\usepackage{amsfonts}
\usepackage{xcolor}
\usepackage[all]{xy}
\usepackage{amssymb, mathrsfs, amsfonts, amsmath}
\usepackage{amsbsy}
\usepackage{amsfonts}
\newtheorem*{acknowledgement}{Acknowledgement}

\newtheorem{lemma}{Lemma}

\newtheorem{remark}{Remark}
\newtheorem{theorem}{Theorem}
\newtheorem{example}{Example}
\numberwithin{equation}{section}

\usepackage{cases}

\usepackage{sidecap}

\makeatletter
\@namedef{subjclassname@2020}{%
\textup{2020} Mathematics Subject Classification}
\makeatother
\begin{document}
	
	\title[Self-similar solutions to the MCF in $\mathbb{R}^{3}$]{Self-similar solutions to the MCF in $\mathbb{R}^{3}$}
	
	\author{Benedito Leandro}
	\author{Rafael Novais}
	\author{Hiuri F. S. dos Reis}
	
	\address[B. Leandro]{Universidade Federal de Goi\'as - UFG, IME, 74690-900, Goi\^ania - GO, Brazil.}
	\email{bleandroneto@ufg.br}
	
	\address[R. Novais]{Universidade Federal de Goi\'as - UFG, IME, 74690-900, Goi\^ania - GO, Brazil.}
	\email{rnovais87@gmail.com}
	
	\address[H. dos Reis]{Instituto Federal de Goiás - IFG, 76400-000, Uruaçu - GO, Brazil.} 
	\email{hiuru.reis@ifg.edu.br}

	\thanks{Rafael Novais was partially supported by PROPG-CAPES [Finance Code 1811476].}
	

	
	\date{\today}
	\begin{abstract} 
        In this paper we make an analysis of self-similar solutions for the mean curvature flow (MCF) by surfaces of revolution and ruled surfaces in $\mathbb{R}^{3}$. We prove that self-similar solutions of the MCF by non-cylindrival surfaces and conical surfaces in $\mathbb{R}^{3}$ are trivial. Moreover, we characterize the self-similar solutions of the MCF by surfaces of revolutions under a homothetic helicoidal  motion in $\mathbb{R}^{3}$ in terms of the curvature of the generating curve. Finally, we characterize the self-similar solutions for the MCF by cylindrical surfaces under a homothetic helicoidal motion in $\mathbb{R}^3$. Explicit families of exact solutions for the MCF by cylindrical surfaces in $\mathbb{R}^{3}$ are also given. 
	\end{abstract}

	\maketitle
	\section{Introduction}\label{intro}
	
	    The mean curvature flow (MCF) is a geometric evolution equation. In other words, it is a way to let submanifolds evolve in a given manifold over time to minimize its volume. Surfaces moving in a self-similar way under the MCF are important in the singularity theory of the flow. In \cite{Halldorsson1,Halldorsson2}, the author gave a complete classification of all self-similar solutions to the curve shortening flow (CSF) in the Euclidean and in the Minkowski plane. Even though several results given about the singularities of the mean curvature flow are known (cf. \cite{Colding} and the references therein), the classification of such solutions for the MCF is significantly harder in higher dimensions, even exhibit exact solutions is quite rare (cf. \cite{Colding,leandro,dosReis,dosReis1}).
	
        Let $M^2$ be an $2$-dimensional manifold and assume that $\widehat{X}^0:M^2\longrightarrow \mathbb{R}^{3}$ smoothly immerses $M^2$ as a hypersurface in the Euclidean space $\mathbb{R}^{3}$. We say that $M_0=\widehat{X}^0(M^2)$ is moved along its mean curvature if there is a whole family $\widehat{X}(\cdot,t)$, $t\in\mathfrak{I}$, of smooth immersions with corresponding hypersurfaces $M_t=\widehat{X}(\cdot,t)(M^2)$ such that
        \begin{eqnarray*}
            \left\{
                \begin{array}{lcc}
                    \displaystyle\frac{\partial \widehat{X}}{\partial t}(p,\,t)=H(p,\,t)N(p,\,t), \quad p\in M^2\\\\
                     \widehat{X}(\cdot,\,0)= \widehat{X}^0.
                \end{array}
            \right.
        \end{eqnarray*}
        Here $H^{t}=H(\cdot,t)$ is the mean curvature and $N^{t}=N(\cdot,t)$ is a unit normal vector field of $M_t$.   
    
        We say that a family of smooth immersions $\widehat{X}(\cdot,t)$ is a self-similar motion (cf. \cite{Halldorsson}) of $M^2$ if $\widehat{X}(p,\,t)=L(t)\widehat{X}(p,\,0)$, where $L(t):\mathbb{R}^{3}\longrightarrow \mathbb{R}^{3}$ is a one parameter family of continuous homotheties such that $L(0)=Id$. Thus,    
        $$\widehat{X}(p,\,t)=\sigma(t)\Gamma(t)\widehat{X}(p,\,0)+\Theta(t);\quad p\in M^2,\,t\in\mathfrak{I}.$$ 
        Here, $\mathfrak{I}$ is an interval containing $0$ and $\sigma:\mathfrak{I}\rightarrow\mathbb{R}$, $\Gamma:\mathfrak{I}\rightarrow SO(3)$ and $\Theta:\mathfrak{I}\rightarrow\mathbb{R}^3$ are differentiable functions such that $\sigma(0)=1,\,\Gamma(0)=Id$ and $\Theta(0)=0$ (cf. \cite{Halldorsson1,Halldorsson}). Under our settings, this self-similar motion is the mean curvature flow of $M^2$ if and only if the equation 
        $$\left\langle\frac{\partial \widehat{X}}{\partial t}(p,\,t),N(p,\,t)\right\rangle=H(p,\,t)$$
        holds for all $p\in M^2$, $t\in\mathfrak{I}$. Considering $t=0$ we have $H=H(\cdot,0)$ and $N=N(\cdot,0)$. Therefore, 
        $$H=\sigma'(0)\langle X,\,N\rangle+\langle \Gamma'(0) X,\,N\rangle+\langle \Theta'(0),\,N\rangle,$$
        where $\sigma$, $\Gamma$ and $\Theta$ stand for dilation, rotation and translation in $\mathbb{R}^{3}$, respectively.

        It is important to say that minimal surfaces (i.e., $H=0$) are {\it trivial} solutions to the MCF. Immediately, we can infer that helicoids, catenoids and planes are trivial solutions to the MCF in $\mathbb{R}^{3}$. These three surfaces are relevant since we are considering surfaces of revolution and ruled surfaces as initial data for the MCF.
    
        The aim of the following theorem is to provide a classification of self-similar solutions for the MFC where the initial data $X$ is a non-cylindrical ruled surface in $\mathbb{R}^{3}$. The grim reaper solution is an example of cylindrical solution for the MCF (cf. Example \ref{grim}). Therefore, it is natural to believe that it is possible to build non-cylindrical solutions too. Note that the helicoid is a non-cylindrical ruled surface, so it is a trivial example for this case. The next theorem was inspired by these examples. We will consider the non-cylindrical ruled surfaces parameterized by lines of striction, i.e., $\langle\beta',w'\rangle=0$ and $w'\neq0$ (cf. \cite{CARMO GD}).

        \begin{theorem}\label{C1}
            Let $X:U\subset \mathbb{R}^2\longrightarrow \mathbb{R}^{3}$ be a non-cylindrical ruled surface in $\mathbb{R}^{3}$, $$X(s,\,u)=\beta(s)+u w(s),$$
            where $\beta(s)$ is a curve in $\mathbb{R}^{3}$ and $w(s)\in T_{\beta(s)}\mathbb{R}^{3}$,\, $|w|=1$. 
            Suppose that $\widehat{X}^{t}(s,\,u)=L(t)X(s,\,u)$ is a self-similar solution to the MCF and $L(t)=\sigma(t)\Gamma(t)+\Theta(t)$ is a homothetic helicoidal motion in $\mathbb{R}^{3}$. Then, $X$ is trivial.
        \end{theorem}
        
           It is important to say that the homothetic helicoidal motion that we are considering is a motion involving rotation, translation and dilation. Nonetheless, we analyse all particular cases of motion. For instance, the translation soliton is a particular and important case (see more in the proof of Theorem \ref{C1}). The helicoidal motion was studied by Halldorsson in \cite[Section 2]{Halldorsson} and inspired this work.
        
         In \cite{leandro}, the authors characterized the solutions for the MCF on the Heisenberg group by ruled surfaces and explicit examples are provided. Moreover, A few months before the first version of this research paper appeared, L\'opez \cite{lopez2021} proved a classification for $\alpha$-self-similar solution for the MCF by
ruled surfaces. He proved that this type of solution must be a cylindrical surface. We can see from \cite[Definition 1.1]{lopez2021} that Theorem \ref{C1} consider a different type of solutions for the MCF in $\mathbb{R}^{3}$ (cf. \cite[Theorem 1.2]{lopez2021}). Moreover, we will provide explicit examples of cylindrical solutions for the MCF in $\mathbb{R}^3$.
   
        Another important type of surfaces in $\mathbb{R}^{3}$ are the surfaces of revolution. In \cite{Halldorsson} the author provided an analysis of self-similar solutions to the MCF by helicoidal surfaces. Here, we will generalize this idea and characterize the self-similar solutions by surfaces of revolution which move by a homothetic helicoidal movement in $\mathbb{R}^3$ (see the definition of this motion in the proof of Theorem \ref{C40} and Theorem \ref{C3}).

        \begin{theorem}\label{C40}
            Let $X:U\subset \mathbb{R}^2\longrightarrow \mathbb{R}^{3}$ be a surface of revolution in $\mathbb{R}^{3}$ satisfying
            \begin{eqnarray}\label{paramrevo} 
                X(u,\,s)=(\phi(s)\cos(u),\,\phi(s)\sin(u),\,\psi(s)),
            \end{eqnarray}
            where $\phi,\,\psi$ are smooth real functions. 
            Consider that $\widehat{X}^{t}(s,\,u)=L(t)X(s,\,u)$ is a self-similar solution to the MCF with initial condition $X$ such that  $L(t)$ is a homothetic helicoidal motion in $\mathbb{R}^{3}$.
            Then $X$ is an initial data for the MCF if and only if the curvature $\kappa$ of $\alpha(s)=(0,\,\phi(s),\,\psi(s))$ is given by
            \begin{eqnarray}\label{caracteriSR}
                \kappa=\frac{1}{|\tau|}\left(2 c\langle\alpha,\,\eta\rangle+ 2b\langle e_3,\,\eta\rangle+\frac{1}{\phi}\langle e_3,\,\tau\rangle\right),
            \end{eqnarray}
            where $\tau=\alpha'$, $\eta=(0,\,\psi',\,-\phi')$ and $e_3=(0,0,1)$. 
        \end{theorem}
        
        The above theorem shows us that to get self-similar solutions for the MCF by surfaces of revolution we need to analyze the curvature $\kappa$ of curves in $\mathbb{R}^2$ satisfying \eqref{caracteriSR}. In \cite{Halldorsson1}, Halldorsson gave us all self-similar solutions to the curve shortening flow in the plane. The condition of such plane curves is similar to \eqref{caracteriSR}. When were plotting the solutions for \eqref{caracteriSR} we realized that some graphs were similar to those in \cite{Halldorsson1}. We will discuss more about this examples of Theorem \ref{C40} in Section \ref{examples}.

        In the next result we will consider the cylindrical ruled surfaces as initial data for the MCF in the Euclidean space. The cylinder itself is an example of ruled surface initial data for a self-similar solution of the MCF in $\mathbb{R}^{3}$. Another important example for this case is the grim reaper solution (cf. \cite{martin}). Inspired by those examples we will characterize all such self-similar solutions by cylindrical surfaces and then, in Section \ref{examples}, we discuss some examples.
        
        \begin{theorem}\label{C3}
            Let $X:U\subset \mathbb{R}^2\longrightarrow \mathbb{R}^{3}$ be a cylindrical ruled surface in $\mathbb{R}^{3}$. 
            If $\widehat{X}^{t}(s,\,u)=L(t)X(s,\,u)$ is a self-similar solution to the MCF with initial condition $$X(s,\,u)=\beta(s)+uw(s),$$ where $L(t)$ is a homothetic helicoidal motion, $\beta$ a plane curve,  $w(s)=(x_0,\,y_0,\,z_0)$ and $|w|=1,$
            then $X$ is given by one of the following surfaces:
            \begin{itemize}
                \item[(I)] $\beta(s)=(0,\,h(s),\,q(s))$ where 
                    \begin{eqnarray*}
                    ax_{0}\langle w,\,\tau\rangle&=&0;\\   2cx_0\langle\beta,\,\eta\rangle+2b\langle w,\,\eta\rangle-2ax_{0}\langle\beta,\,\tau\rangle&=&\dfrac{x_{0}|\tau|^3\kappa}{|\tau|^2 -\langle w,\,\tau\rangle^{2}}. 
                    \end{eqnarray*}
                \item[(II)] $\beta(s)=(h(s),\,q(s),\,0)$ where 
                    \begin{eqnarray*}
                        a[h'-x_{0}\langle w,\,\tau\rangle]&=&0;\\   2cz_0 \langle\beta,\,\eta\rangle +2bz_{0}\langle e_1,\,\eta\rangle+2aq\langle w,\,\eta\rangle&=&\dfrac{z_{0}|\tau|^3\kappa}{|\tau|^2 -\langle w,\,\tau\rangle^{2}}. 
                    \end{eqnarray*} 
            \end{itemize}
            Here, $a$, $b$ and $c$ are constants related with rotation, translation and dilation of the surface $X$ in the Euclidean space. Moreover, $\tau$, $\eta$ and $\kappa$ are, respectively, the tangent, the normal and the curvature of $\beta.$
        \end{theorem}

        \begin{remark}
            It is important to say that a particular case of the first item of Theorem \ref{C3} was already provided in \cite{Halldorsson1}. In fact, considering that $\beta$ is parameterized by the arc length such that $\langle w,\,\tau\rangle=0$, we get the curves given by Halldorsson in \cite{Halldorsson1}, see equation (2.3). These curves are solutions for the curve shortening flow (CSF) in $\mathbb{R}^2$ and they were already classified by him. In Section \ref{examples} we will prove explicit and non-trivial solutions for Item (I). The grim reaper is one of those explicit examples. 
            
            In Item (II) the solutions that do not rotate ($a=0$) such that $\langle w,\,\tau\rangle=0$ also are particular cases of \cite{Halldorsson1}, where $\beta$ is parameterized by the arc length. Moreover, the only non-trivial solutions must have $a=0$, for $\langle w,\,\tau\rangle$ null or not. We will discuss this in Section \ref{examples}.
        \end{remark}

        Let us briefly discuss the organization of this manuscript. In Section 2 we will remember the geometry of ruled surfaces and surfaces of revolution. In Section 3 we present the proof of the main results. Then, in Section 4 we first provide numerical solutions  of Theorem \ref{C40}. Finally, we discuss examples of Theorem \ref{C3}, showing some cases of trivial and exact solutions. 
	\section{Background}\label{secBack}

        \subsection{Ruled surfaces in $\mathbb{R}^3$}\label{sec3}
        
            Let $M^2\subset\mathbb{R}^3$ be a noncylindrical ruled surface (cf. \cite{CARMO GD}). Therefore, $M^2$ is given by 
            \begin{equation}\label{regrada R}
                X(s,u)=\beta(s)+u w(s),
            \end{equation}
            where $\beta$ is a curve of $\mathbb{R}^3$ and $w(s)\in T _ {\beta (s)}\mathbb{R}^3$, $|w|=1$, is a vector such that $\langle\beta', w'\rangle = 0$. The assumption that $M^2\subset\mathbb{R}^3$ is noncylindrical is expressed by $w'(t)\neq0$ for all $t$. In this case we consider $\langle w',\,w' \rangle=1$.
            
            Let us establish some notation:
            $$X_s = \beta'(s)+u\cdot w'(s),\quad\quad X_u = w(s),\quad\quad X_s\wedge X_u=\beta'\wedge w+u\cdot w'\wedge w.$$
            Thus, $\langle w',w\rangle=\langle w',\beta'\rangle=0$, and then
            \begin{eqnarray}\label{lambda1}
                \beta'\wedge w=\lambda w' 
            \end{eqnarray}
            for some function $\lambda=\lambda(s)$. The function $\lambda$ is called the {\it distribution parameter}. 
            We can rewrite
            $$X_s\wedge X_u=\lambda w'+u\cdot w'\wedge w,$$ and hence
            \begin{eqnarray*}
                EG-F^2 = |X_s\wedge X_u|^2 = \lambda^2+u^2.\\
            \end{eqnarray*}     
            
            The unitary normal vector field over $X(s,u)$ is given by 
            \begin{eqnarray}\label{normalvectorfield}
                N=\frac{X_s\wedge X_u}{|X_s\wedge X_u|}=\frac{\lambda w'+u\cdot w'\wedge w}{\sqrt{\lambda^2+u^2}}.
            \end{eqnarray}
            Moreover, the coefficients of the first fundamental form are 
            \begin{align*}
                E = \langle X_{s},X_{s}\rangle = |\beta'|^2+u^2; && F = \langle X_{s},X_{u}\rangle = \langle \beta',w\rangle; &&  G = \langle X_{u},X_{u}\rangle = 1.
            \end{align*}     
            Furthemore, $$X_{ss}=\beta''+u\cdot w'',\quad\quad X_{su}=w', \quad\quad X_{uu}=0.$$
            Thus, the coefficients of the second fundamental form are
            \begin{align*}
                g = \langle N,X_{uu}\rangle = 0;  && f = \langle N,X_{su}\rangle = \frac{\lambda}{\sqrt{\lambda^2+u^2}};
            \end{align*}     
            \begin{eqnarray*}
                e &=& \langle N,X_{ss}\rangle = \frac{1}{\sqrt{\lambda^2+u^2}}(\lambda\langle w',\beta''\rangle+\lambda\cdot u\langle w',w''\rangle-u\langle w\wedge w',\beta''\rangle-u^2\langle w\wedge w',w''\rangle).\\
            \end{eqnarray*}  
            
             Considering the orthonormal frame $\{w,w',w\wedge w'\}$ in $\mathbb{R}^3$ we have
            \begin{eqnarray}\label{beta'}
                \beta' = \langle \beta',w\rangle w+\langle \beta',w'\rangle w'+\langle \beta',w\wedge w'\rangle w\wedge w'= F w+\lambda w\wedge w'.
            \end{eqnarray}       
            Since $\langle\beta',w'\rangle=0$ and from the fact that $w$ is unitary, we get 
            \begin{eqnarray*}
                \langle\beta'',w'\rangle = -\langle\beta',w''\rangle = F-\lambda J,
            \end{eqnarray*} 
            where $J:=\langle w\wedge w', w''\rangle$. Taking the derivative of
            \begin{eqnarray}\label{complambda}
                \lambda=\langle\beta'\wedge w,\,w'\rangle
            \end{eqnarray}
            we have 
            $$\beta''\wedge w+\beta'\wedge w' = \lambda'w'+\lambda w''.$$ Therefore, since $|w'|=1$ we get
            \begin{eqnarray*}
                \langle\beta''\wedge w,w'\rangle&=& \lambda'.
            \end{eqnarray*}      
             We can conclude that $$e=\frac{\lambda(F-\lambda J)-u\lambda'-u^2J}{\sqrt{\lambda^2+u^2}}.$$
             
            Finally, we get the gaussian and the mean curvature, respectively,
            \begin{equation}\label{gaussiana r3}
                K=\frac{eg-f^2}{EG-F^2}=\frac{-\lambda^2}{(\lambda^2+u^2)^2}
            \end{equation} 
            and
            \begin{equation}\label{media r3}
                H=\frac{Eg-2fF+Ge}{2(EG-F^2)}=-\frac{\lambda F+\lambda^2J+u\lambda'+u^2J}{2(\lambda^2+u^2)^{3/2}}.
            \end{equation}
        \subsection{Surfaces of Revolution in $\mathbb{R}^3$}
        
        Let $X:U\subset \mathbb{R}^2\longrightarrow \mathbb{R}^{3}$ be a revolution surface in $\mathbb{R}^{3}$ satisfying $$X(u,\,s)=(\phi(s)\cos(u),\,\phi(s)\sin(u),\,\psi(s)),$$  where $\phi,\,\psi$ are smooth real functions.

        We start by providing the normal vector field:
        \begin{eqnarray}\label{Nrevolution}
            N(s,\,u)=\dfrac{1}{\sqrt{\phi'^2+\psi'^2}}(\psi'\cos u,\,\psi'\sin u,\, -\phi')
        \end{eqnarray}

        From a straightforward computation we get the coefficients of the first and second fundamental formulas.
        \begin{align*}
            E=\phi^{2}; && F = 0; &&  G = \phi'^2 + \psi'^2;
        \end{align*}   
        \begin{align*}
            e=\dfrac{-\phi\psi'}{\sqrt{\phi'^2 + \psi'^2}};&& f = 0;\,&&  g=\dfrac{\phi''\psi'-\phi'\psi''}{\sqrt{\phi'^2 + \psi'^2}}.
        \end{align*} 
        
        Then, the gaussian and the mean curvature are given by:
        \begin{equation*}
            K=\frac{-\psi'^2\phi''+\phi'\psi'\psi''}{\phi(\phi'^2 + \psi'^2)^2}
        \end{equation*} and
        \begin{equation}
            H=\dfrac{\phi(\phi''\psi'-\phi'\psi'')-\psi'(\phi'^2 + \psi'^2)}{2\phi(\phi'^2 + \psi'^2)^{3/2}}.\label{media r31}
        \end{equation}
        
        It is essential to remember that Efimov's theorem for surfaces of revolution says that there is no regular complete surface of revolution in $\mathbb{R}^3$ with $K\leq\delta<0,$ for some real constant $\delta>0$.
    \section{Proof of the Main Results}

        \begin{lemma}\label{T3}
            Let $X:U\subset \mathbb{R}^2\longrightarrow \mathbb{R}^{3}$ be a noncylindrical ruled surface in $\mathbb{R}^{3}$ satisfying \eqref{regrada R}. Then $\widehat{X}^{t}(s,\,u)=L(t)X(s,\,u)$ is a self-similar solution to the MCF on $\mathbb{R}^{3}$ if and only if
            \begin{equation}\label{lemma1}\
                \begin{cases}
                    \langle \Gamma'(0)w, w'\wedge w\rangle=0;\\
                    \lambda\langle \Gamma'(0)w, w'\rangle+\langle \Gamma'(0)\beta, w'\wedge w\rangle+\langle \Theta'(0), w'\wedge w\rangle+\sigma'(0)\langle \beta, w'\wedge w\rangle=0;\\
                    \lambda(\langle \Gamma'(0)\beta, w'\rangle+\langle \Theta'(0), w'\rangle+\sigma'(0)\langle \beta,\,w'\rangle)+\dfrac{1}{2}\langle w\wedge w', w''\rangle=0;\\
                    \lambda\langle \beta',w\rangle=0.
                \end{cases}
            \end{equation}         
            where $L(t)=\sigma(t)\Gamma(t)+\Theta(t)$ is a homothetic helicoidal motion in $\mathbb{R}^{3}$. Here $\sigma$, $\Gamma$ and $\Theta$ stand for the dilation, rotation and translation in $\mathbb{R}^{3}$, respectively. Moreover, $\lambda=\langle\beta'\wedge w,\,w'\rangle$ must be constant.
        \end{lemma}
        
        \begin{proof}[Proof of Lemma \ref{T3}] 
            Let $X:U\subset \mathbb{R}^2\longrightarrow \mathbb{R}^{3}$ be a noncylindrical ruled surface in $\mathbb{R}^{3}$ satisfying \eqref{regrada R}.
            Consider a one parameter family of surfaces $\widehat{X}^{t}(s,u)=L(t)X(s,u)$, where $L(t)$ is an homothetic helicoidal motion of $\mathbb{R}^3$.
            We can rewrite $\widehat{X}$ in the following form
            \begin{equation*}
                \widehat{X}^{t}(s,u)=\sigma(t)\Gamma(t)X(s,u)+\Theta(t), 
            \end{equation*}
            where $\Gamma$ and $\Theta$ stand for rotations and translations in $\mathbb{R}^{3}$, respectively. Here, the dilation function $\sigma$ determines the scaling and $\sigma(0)=1.$
            That being said, $\widehat{X}^{\tau}$ is a solution to the MCF in $\mathbb{R}^3$ if and only if
            $$\begin{cases}
                \dfrac{\partial\widehat{X}^{t}}{\partial t}=H^{t}(s,u)N^{t}(s,u);\\\\
                \widehat{X}^{0}(s,u)= X(s,u),\\
            \end{cases}$$  
            for all $t\in \mathfrak{I}\subset\mathbb{R}.$
        
            Considering that $\widehat{X}$ is a solution for the MCF, at $t=0$ we have
            \begin{equation}\label{soliton r3}
                \langle \sigma'(0)\Gamma(0)X+\Gamma'(0)X+\Theta'(0),N\rangle=H.
            \end{equation}       
            Therefore, from \eqref{normalvectorfield} we get
            \begin{eqnarray*}
                \langle \Gamma'(0)X,N\rangle &=& \langle \Gamma'(0)(\beta+uw),\dfrac{X_s\wedge X_u}{|X_s\wedge X_u|}\rangle\\
                                             &=& (EG-F^2)^{-1/2}\langle \Gamma'(0)(\beta+u w),\lambda w'+uw'\wedge w\rangle\\
                                             &=& (EG-F^2)^{-1/2}(\lambda\langle \Gamma'(0)\beta, w'\rangle+u\lambda\langle \Gamma'(0)w,w'\rangle\nonumber\\
                                             &+&u\langle\Gamma'(0)\beta,w'\wedge w\rangle+u^2\langle\Gamma'(0)w,w'\wedge w\rangle),
                \end{eqnarray*}
            and
            \begin{eqnarray*}
                \sigma'(0) \langle \Gamma(0)X,N\rangle &=&\sigma'(0) \langle \Gamma(0)(\beta+uw),\dfrac{X_s\wedge X_u}{|X_s\wedge X_u|}\rangle\\
                                                       &=& \sigma'(0)(EG-F^2)^{-1/2}\langle \Gamma(0)(\beta+u w),\lambda w'+uw'\wedge w\rangle\\
                                                       &=& \sigma'(0)(EG-F^2)^{-1/2}(\lambda\langle \Gamma(0)\beta, w'\rangle+u\lambda\langle \Gamma(0)w,w'\rangle\nonumber\\
                                                       &+&u\langle\Gamma(0)\beta,w'\wedge w\rangle+u^2\langle\Gamma(0)w,w'\wedge w\rangle)\\
                                                       &=& \sigma'(0)(EG-F^2)^{-1/2}(\lambda\langle \beta, w'\rangle +u\langle\beta,w'\wedge w\rangle),  
            \end{eqnarray*}
            where $\beta:=\beta(s)$ and $w:=w(s)$ are given by \eqref{regrada R}. 
            
            Let us define the following functions of $s$, in which $|w|=1$:
            \begin{eqnarray*}
                V &:=& \langle \Gamma'(0)\beta,\, w'\rangle;\quad
                W := \langle \Gamma'(0)w,\, w'\rangle;\quad
                Y := \langle \Gamma'(0)\beta,\, w'\wedge w\rangle;\\
                Z &:=& \langle \Gamma'(0)w,\, w'\wedge w\rangle;\quad
                C := \langle \Theta'(0), w'\rangle;\quad
                D := \langle \Theta'(0),\, w'\wedge w\rangle;\\
                A &:=& \sigma'(0)\langle \beta, w'\wedge w\rangle;\quad
                B:= \sigma'(0)\langle \beta,\,w'\rangle.
            \end{eqnarray*}   
            Thus, we get
            \begin{eqnarray}\label{theq1}
                \langle \Gamma'(0)X,N\rangle=\frac{\lambda V+u\lambda W+uY+u^2Z}{\sqrt{\lambda^2+u^2}}
            \end{eqnarray}
            and 
            \begin{eqnarray}\label{theq111}
            \langle \Gamma(0)X,N\rangle=\frac{\lambda B+uA}{\sqrt{\lambda^2+u^2}}.
            \end{eqnarray}
            Moreover, 
           \begin{eqnarray*}
                \langle \Theta'(0),N\rangle &=& \langle \Theta'(0),\dfrac{X_s\wedge X_u}{|X_s\wedge X_u|}\rangle\\
                                            &=& (EG-F^2)^{-1/2}\langle \Theta'(0),\lambda w'+u\cdot w'\wedge w\rangle\\
                                            &=& (EG-F^2)^{-1/2}(\lambda\langle \Theta'(0), w'\rangle+u\langle \Theta'(0), w'\wedge w\rangle).                         
            \end{eqnarray*}
            That is,
            \begin{eqnarray}\label{theq2}
                \langle \Theta'(0),N\rangle=\frac{\lambda C+u D}{\sqrt{\lambda^2+u^2}}.  
            \end{eqnarray}
            Combining  \eqref{media r3}, \eqref{soliton r3}, \eqref{theq1}, \eqref{theq111} and \eqref{theq2} we obtain the following equation:
            \begin{eqnarray*}
                \frac{\lambda V+u\lambda W+uY+u^2Z}{\sqrt{\lambda^2+u^2}}+\frac{\lambda C+u D}{\sqrt{\lambda^2+u^2}}+\frac{\lambda B+uA}{\sqrt{\lambda^2+u^2}}=-\frac{\lambda     F+\lambda^2J+u\lambda'+u^2J}{2(\lambda^2+u^2)^{3/2}},
            \end{eqnarray*}  
            
            A straightforward computation can prove that the above equation is a fourth order polynomial in $u.$ That is,
            \begin{multline}
                 2u^4Z+2u^3(\lambda W+Y+D+A)+2u^2(\lambda^2Z+\lambda(V+C+B)+\dfrac{1}{2}J)\nonumber\\
                 +u(2\lambda^3W+2\lambda^2(Y+D+A)+\lambda')+2\lambda^3(V+C+B)+\lambda F+\lambda^2J = 0. \label{polinomio2}
            \end{multline}   
            From the above equation we can gather that
            \begin{equation}\label{sistema r}
                \begin{cases}
                    Z=0;\\
                    \lambda W=-(Y+D+A);\\
                    \lambda(V+C+B)=-\dfrac{1}{2}J;\\
                    \lambda'=0;\\
                    \lambda F=0.  
                \end{cases}
            \end{equation} 
        
            Therefore,    
            \begin{equation*}\
                \begin{cases}
                    \langle \Gamma'(0)w, w'\wedge w\rangle=0;\\
                    \lambda\langle \Gamma'(0)w, w'\rangle+\langle \Gamma'(0)\beta, w'\wedge w\rangle+\langle \Theta'(0), w'\wedge w\rangle+\sigma'(0)\langle \beta, w'\wedge w\rangle=0;\\
                    \lambda(\langle \Gamma'(0)\beta, w'\rangle+\langle \Theta'(0), w'\rangle+\sigma'(0)\langle \beta,\,w'\rangle)+\dfrac{1}{2}\langle w\wedge w', w''\rangle=0;\\
                    \lambda'=0;\\
                    \lambda\langle \beta',w\rangle=0.   
                \end{cases}
            \end{equation*}  
        
            Conversely, suppose that \eqref{sistema r} is satisfied. Then,    
            \begin{eqnarray*}
                \langle L'(0)X,N\rangle-H &=& \langle \sigma'(0) \Gamma(0)X+\Gamma'(0)X+\Theta'(0),\,N\rangle-H\\
                                          &=& \frac{\lambda B+uA+\lambda V+u\lambda W+uY+u^2Z+\lambda C+u D}{\sqrt{\lambda^2+u^2}}\\
                                          &\;&+\frac{\lambda F+\lambda^2J+u\lambda'+u^2J}{2(\lambda^2+u^2)^{3/2}}\\
                                          &=& \frac{\lambda(V+C+B)+u(Y+D+A)+\lambda uW}{\sqrt{\lambda^2+u^2}}+\frac{\lambda^2J+u^2J}{2(\lambda^2+u^2)^{3/2}}\\
                                          &=& \frac{-\dfrac{1}{2}J-\lambda uW+u\lambda W}{\sqrt{\lambda^2+u^2}}+\frac{\lambda^2J+u^2J}{2(\lambda^2+u^2)^{3/2}}\\  
                                          &=& \frac{-2(\lambda^2+u^2)\dfrac{1}{2}J+(\lambda^2+u^2)J}{2(\lambda^2+u^2)^{3/2}}\\  
                                          &=& 0.
            \end{eqnarray*}    
    \end{proof}    
    
      \begin{theorem}\label{Sol Trivial}
            Let $X:U\subset \mathbb{R}^2\longrightarrow \mathbb{R}^{3}$ be a non-cylindrical ruled surface in $\mathbb{R}^{3}$, $$X(s,\,u)=\beta(s)+u w(s),$$
            where $\beta(s)$ is a curve in $\mathbb{R}^{3}$ and $w(s)\in T_{\beta(s)}\mathbb{R}^{3}$,\, $|w|=1$. 
            Suppose that $\widehat{X}^{t}(s,\,u)=L(t)X(s,\,u)$ is a self-similar solution to the MCF and $L(t)=\sigma(t)\Gamma(t)+\Theta(t)$ is a homothetic helicoidal motion in $\mathbb{R}^{3}$. If $w\subset\mathbb{S}^{2}$ is a geodesic, then $X$ must be trivial.
        \end{theorem}       
    \begin{proof}[Proof of Theorem \ref{Sol Trivial}]    
        By hypothesis, $X$ is a non-cylindrical ruled surface satisfying \eqref{regrada R}. Hence, $w$ is a curve in $\mathbb{S}^2$ parametrized by the arc length.
        
        Considering $\lambda=0$, from \eqref{gaussiana r3} we conclude that the Gaussian curvature $K$ of $X$ vanishes identically, so $X$ is a plane.
        
        Now, assume that $\lambda\neq0.$ Note that $J=\langle w\wedge w', w''\rangle$ is the geodesic curvature of $w$ as a curve of sphere $\mathbb{S}^2$. If $J=0$, then $w$ is a geodesic of $\mathbb{S}^2$. Therefore, $w''$ is a linear combination of $w$ and $w'$. Since $\langle w'', w\rangle=-1$ and $\langle w'', w'\rangle=0$, we have $w''=-w$. On the other hand, combining \eqref{beta'} with the last equation of \eqref{lemma1} we obtain $$\beta'=\lambda w\wedge w'.$$ Taking the derivative of the above equation we get 
         \begin{eqnarray*}
            \beta''&=&\lambda (w'\wedge w'+w\wedge w'')\\
            &=&\lambda (w'\wedge w'- w\wedge w)\\
                   &=& 0.
         \end{eqnarray*}
        We conclude that $\beta$ must be a straight line and orthogonal to $w$. Therefore, $X$ is a helicoid.
    \end{proof}
    \begin{lemma}\label{Triedo de Darboux}
        Let $X:U\subset \mathbb{R}^2\longrightarrow \mathbb{R}^{3}$ be a non-cylindrical ruled surface in $\mathbb{R}^{3}$ satisfying \eqref{regrada R}. Then, the set of vector fields $\{w,\,w',\, w\wedge w'\}$ satisfies the Darboux equations, i.e.,
        \begin{eqnarray*}
            \begin{cases}
                w''=J w'\wedge w - w;\\
                (w\wedge w')'=-J w',
            \end{cases}
        \end{eqnarray*} 
        were $J = \langle w\wedge w', w''\rangle$.    
    \end{lemma}     

    \begin{proof}[Proof of Lemma \ref{Triedo de Darboux}]
        Considering the orthonormal frame $\{w,w',w\wedge w'\}$ in $\mathbb{R}^3$, we have
        \begin{eqnarray}\label{bene1}
             w''=\rho_1 w + \rho_2 w' + \rho_3 w\wedge w'.  
        \end{eqnarray}
        Then, by similar arguments used in Theorem \ref{Sol Trivial} we conclude that
        \begin{align*}
            \rho_1=-1; && \rho_2=0; && \rho_3=J.
        \end{align*}
        Moreover, from \eqref{bene1} we get
        \begin{eqnarray*}
            (w\wedge w')'&=&w'\wedge w'+w\wedge w'';\\
                         &=&w\wedge (J w'\wedge w - w);\\
                         &=& J w\wedge(w\wedge w')- w\wedge w; \\
                         &=& J w\langle w,w'\rangle-Jw'\langle w,w\rangle; \\
                         &=& -Jw'.
             \end{eqnarray*}  
             
        Note that, $w\subset\mathbb{S}^2$ is a curve parametrized by arc length. Therefore, at $p=w(s)$ we can infer that $\{w'(s),\,-w(s),\, w(s)\wedge w'(s)\}$ is a Daboux trihedron.     
    \end{proof}
    \begin{proof}[Proof of Theorem \ref{C1}]    
        We will write $\beta$ and $w$ in coordinates, then by Lemma \ref{T3} we will explicitly provide $\beta$ and $w$. 
        
     From the proof of Theorem \ref{Sol Trivial}, we already know that if $\lambda=0$ then $X$ is trivial. So, from now on assume $\lambda\neq0.$ From Lemma \ref{T3} we have
        \begin{eqnarray}\label{sist111}
            \begin{cases}
                \langle \Gamma'(0)w, w'\wedge w\rangle=0;\\
                \lambda\langle \Gamma'(0)w, w'\rangle+\langle L'(0)\beta, w'\wedge w\rangle=0;\\
                J=-2\lambda\langle L'(0)\beta, w'\rangle;\\
                \langle \beta',w\rangle=0.   
            \end{cases}
        \end{eqnarray}
        Here, $J = \langle w\wedge w', w''\rangle$ is the geodesic curvature of $w\subset\mathbb{S}^2$ and $\lambda$ is constant. 
        
        Since $L=\sigma\Gamma+\Theta$ is a homothetic helicoidal motion in $\mathbb{R}^{3}$, without loss of generality, we can describe it by
        $$\Gamma(\tau)=
        \begin{pmatrix}
             \cos{\xi(\tau)} & -\sin{\xi(\tau)} &  0\\ 
             \sin{\xi(\tau)} &  \cos{\xi(\tau)} &  0\\ 
                           0 &                0 &  1\\ 
        \end{pmatrix}\quad\mbox{and}\quad \Theta(\tau)=(0,0,\zeta(\tau))$$
        such that 
        \begin{eqnarray}\label{isometriatheo2}
            \Gamma'(0)=a
            \begin{pmatrix}
                0 & -1 &  0\\ 
                1 &  0 &  0\\ 
                0 &  0 &  0\\ 
            \end{pmatrix}
            \quad\mbox{and}\quad\Theta'(0)=b(0,0,1),
        \end{eqnarray}
        where $a=\xi'(0)$, $b=\zeta'(0)$ and $c=\sigma'(0)$. Here, we are assuming $\xi(0)=0$. By Halldorsson  \cite[Section 2]{Halldorsson}, we can consider $\Theta$ parallel to the axis of rotation.     
        
        Suppose the family $\widehat{X}^{\tau}(s,u)=L(\tau)X(s,u)$ is a solution to the MFC on $\mathbb{R}^3$ with initial condition $X(s,u)=\beta(s)+u w(s)$, where  $\beta(s)=(x_1(s),x_2(s),x_3(s))$ and $w(s)=(y_1(s),y_2(s),y_3(s)).$      
        
        Since $|w|=1$, from the first equation of \eqref{sist111} we have
        \begin{eqnarray}\label{louc}
            0 &=& \langle \Gamma'(0)w, w'\wedge w\rangle\nonumber\\
              &=& a\{-y_2(y'_2y_3-y_2y'_3)+y_1(-y'_1y_3+y_1y'_3)\}\nonumber\\  
              &=& a\{y_3(-y_2y'_2-y_1y'_1)+y'_3(y^2_1+y^2_2)\}\nonumber\\  
              &=& a\{y^2_3y'_3+y'_3(1-y^2_3)\}\nonumber\\                      
              &=& ay'_3.
        \end{eqnarray}  
        Thus, either $a = 0$ or $y'_3 = 0$.
        
  Now, we must consider three possibilities:
        \begin{align*}
            \textbf{(I)}\quad a\neq0 \quad\text{and}\quad y'_3=0; && \textbf{(II)}\quad a=0 \quad\text{and}\quad y'_3=0; && \textbf{(III)}\quad a=0 \quad\text{and}\quad y'_3\neq0.
        \end{align*}
        
        \noindent{\bf Case (I)}: $a\neq0$ and $y'_3=0$.

      Under our settings, we can conclude that $$\langle w',e_3\rangle=0.$$    
        However, $y_3$ is constant, thus $w$ must be a flat curve in $\mathbb{R}^{3}$ and orthogonal to the axis of rotation $e_3$. 
        
        Since $| w | = 1$, locally $w$ is a circle of radius $r\leq1$ in $\mathbb{S}^2$, given by 
        \begin{eqnarray}\label{w}
            w(s)=\bigg( r\cos{(s/r)},r\sin{(s/r)},\sqrt{1-r^2} \bigg).
        \end{eqnarray}
        Hence, a straightforward computation gives us
        \begin{eqnarray*}
            w' &=& ( -\sin{(s/r)},\cos{(s/r)},0);\\\\
            w'\wedge w &=& \bigg(\sqrt{1-r^2}\cos{(s/r)},\sqrt{1-r^2}\sin{(s/r)},-r\bigg);\\\\
            w''&=& \bigg( -\dfrac{1}{r}\cos{(s/r)},-\dfrac{1}{r}\sin{(s/r)},0\bigg);\\\\
            \langle w\wedge w', w''\rangle&=&\dfrac{\sqrt{1-r^2}}{r}.
        \end{eqnarray*}       
        
        On the other hand, if $r=1$, then $X$ is a helicoid, i.e., a trivial solution for the MCF. In fact, when $r=1$ we have $J=\langle w\wedge w', w''\rangle=0$ and from Theorem \ref{Sol Trivial}, $X$ is a helicoid.
         
        Considering $r<1$, from \eqref{isometriatheo2} we have
        \begin{eqnarray*}
            \Gamma'(0)w &=& (-ar\sin{(s/r)},ar\cos{(s/r)},0);
        \end{eqnarray*}        
        \begin{eqnarray*}
            L'(0)\beta &=&\sigma'(0) \Gamma(0)\beta+\Gamma'(0)\beta+\Theta'(0) = (cx_1-ax_2,cx_2+ax_1,cx_3+b).
        \end{eqnarray*}         
        Then, from the second and third equations of the system \eqref{sist111} we have, respectively,
        \begin{eqnarray}\label{sist w 1}
            0 &=& \lambda\langle \Gamma'(0)w, w'\rangle+\langle L'(0)\beta, w'\wedge w\rangle= (\lambda a-cx_3-b)r\nonumber\\
              &\;&+\sqrt{1-r^2}[(c\cos(s/r)+a\sin(s/r))x_1+(c\sin(s/r)-a\cos(s/r))x_2]
        \end{eqnarray}    
        and 
        \begin{eqnarray}\label{sist w 2}
            0 &=& \lambda(\langle L'(0)\beta, w'\rangle+\dfrac{1}{2}\langle w\wedge w', w''\rangle\nonumber\\
              &=& \lambda[(a\cos(s/r)-c\sin(s/r))x_1+(a\sin(s/r)+c\cos(s/r))x_2]+\dfrac{\sqrt{1-r^2}}{2r}.
        \end{eqnarray}     
        Moreover, the last equation of the system \eqref{sist111} yields to
        \begin{eqnarray}\label{sist w 3}
        0 &=& \langle \beta',w\rangle\nonumber\\
          &=& rx'_1\cos{(s/r)}+rx'_2\sin{(s/r)}+x'_3\sqrt{1-r^2}.
        \end{eqnarray}         
        
        We can rewrite equations \eqref{sist w 1}, \eqref{sist w 2} and \eqref{sist w 3} to obtain    
        \begin{eqnarray}\label{sqv1}
            (c\cos(s/r)+a\sin(s/r))x_1+(c\sin(s/r)-a\cos(s/r))x_2 &=& \dfrac{r(cx_3+b-\lambda a)}{\sqrt{1-r^2}};
        \end{eqnarray}
        \begin{eqnarray}\label{sqv2}
            -(c\sin(s/r)-a\cos(s/r)) x_{1}+(c\cos(s/r)+a\sin(s/r))x_{2} &=& -\dfrac{\sqrt{1-r^2}}{2r\lambda}.
        \end{eqnarray}     
        \begin{eqnarray}\label{sqv3}
            x'_1\cos{(s/r)}+x'_2\sin{(s/r)} &=& -\dfrac{x'_3\sqrt{1-r^2}}{r}
        \end{eqnarray}

        Now, taking the derivative of \eqref{sqv2} and combining with \eqref{sqv1} we get    
        \begin{eqnarray*}
             x'_1[a\cos(s/r)-c\sin(s/r)] + x'_2[a\sin(s/r)+c\cos(s/r)]= \dfrac{(cx_3+b-\lambda a)}{\sqrt{1-r^2}}.
        \end{eqnarray*}  
       Considering the non-cylindrical ruled surfaces parameterized by lines of striction, we have $\langle \beta',\,w'\rangle=0$, combining this fact with \eqref{sqv3} and the above equation, we get       
        \begin{eqnarray*}
             \dfrac{(cx_3+b-\lambda a)}{\sqrt{1-r^2}}+\dfrac{ax'_3\sqrt{1-r^2}}{r}=0.
        \end{eqnarray*}
            So, 
        \begin{eqnarray}\label{da1}
            x'_3=-\dfrac{r(cx_3+b-\lambda a)}{a(1-r^2)}.
        \end{eqnarray}   
        
        From \eqref{lambda1}, we have $\beta'\wedge w=\lambda w' $. Thus,
        \begin{equation*}
            \lambda=\langle \beta'\wedge w, w'\rangle=x'_{3}r-\left(\sqrt{1-r^2}\right)[x'_1\cos(s/r)+x'_2\sin(s/r)].
        \end{equation*}      
        Combining equations $\eqref{sqv3}$ and \eqref{da1} with the $\lambda$ expression above we obtain        
        \begin{eqnarray}\label{lambda com x'_3}
            \lambda &=& x'_{3}r-\left(\sqrt{1-r^2}\right)\left[-\dfrac{x'_3\sqrt{1-r^2}}{r}\right]\nonumber\\
                    &=& \dfrac{x'_3}{r}.
        \end{eqnarray}           
        Then,
        \begin{equation}\label{lambda}
            \lambda=\dfrac{cx_3+b}{ar^2}.
        \end{equation}
        Since $\lambda$ is constant, we must have either $c=0$ or $c\neq0$ with $x_3(s)$ constant.
            
        If $c\neq0$, we must have $x_3(s)$ constant. Thus, from \eqref{da1}  we have  
        \begin{eqnarray*}
            x_3(s)=\dfrac{-b+\lambda a}{c}.
        \end{eqnarray*}         
        Substituting in equation \eqref{lambda} we get
        \begin{eqnarray*}
            (1-r^2)\lambda=0,
        \end{eqnarray*}
        which is a contradiction, since $r<1$ and $\lambda\neq0$.
        
        If $c=0$, from \eqref{lambda} we have 
        $$\lambda = \dfrac{b}{ar^2}$$
         and from \eqref{da1} we can infer that $$x'_3 = \dfrac{b}{ar^2}.$$
        On the other hand, by equation \eqref{lambda com x'_3} we have  $$(1-r)b=0.$$
        As $r<1$ we must have $b=0$, but this implies $\lambda=0$ which is a contradiction.
         
         Therefore, we can conclude that Case (I) can not occur.\\
        
        \noindent{\bf Case (II)}: $a=0$ and $y'_3=0$.
        
        Analogously to the previous case, when $y'_3=0$, $w$ will be given by \eqref{w}. Moreover, we can consider $r<1$. Taking $a=0$ in \eqref{isometriatheo2}, we have $\Gamma'(0)=0$. Therefore, we can rewrite \eqref{sist111} in the following form:
        \begin{eqnarray}\label{sist111 a=0}
            \begin{cases}
                \langle L'(0)\beta, w'\wedge w\rangle=0;\\
                \lambda\langle L'(0)\beta, w'\rangle+\dfrac{1}{2}\langle w\wedge w', w''\rangle=0;\\
                \langle \beta',w\rangle=0.   
            \end{cases}
        \end{eqnarray}         
        From the first and second equations of \eqref{sist111 a=0} we have, respectively,
        \begin{eqnarray}\label{sist a=0 1}
            0 &=& \langle L'(0)\beta, w'\wedge w\rangle\nonumber\\
              &=& c\sqrt{1-r^2} (x_1\cos(s/r) +  x_2\sin(s/r))-r(cx_3+b);
        \end{eqnarray}   
      and
        \begin{eqnarray}\label{sist a=0 2}
            0 &=& \lambda\langle L'(0)\beta, w'\rangle+\dfrac{1}{2}\langle w\wedge w', w''\rangle\nonumber\\
              &=& 2c\lambda(-x_1\sin(s/r) + x_2\cos(s/r)+\dfrac{\sqrt{1-r^2}}{r}.
        \end{eqnarray}      
        Moreover, the last equation of \eqref{sist111 a=0} gives us
        \begin{eqnarray}\label{sist a=0 3}
            0 &=& \langle \beta',w\rangle\nonumber\\
              &=& x'_1 r\cos(s/r) +  x'_2 r\sin(s/r)+x'_3 \sqrt{1-r^2}.
        \end{eqnarray}       
        
        We can assume that $c\neq0$. Otherwise, from \eqref{sist a=0 2} we obtain $\sqrt{1-r^2}=0$, which implies that $r=1$.
        
        
        and

        Now, take the first derivative of \eqref{sist a=0 2} and combine it with $\langle\beta',w'\rangle=0$ to obtain
        \begin{eqnarray*}
            x_1\cos(s/r)+x_2\sin(s/r)&=&0.
        \end{eqnarray*}
      Then, from the above equation and \eqref{sist a=0 1} we have
        \begin{eqnarray*}
            x_3 &=& -\dfrac{b}{rc}
        \end{eqnarray*}        
       Thus, from \eqref{sist a=0 3} we get $$x'_1 \cos(s/r) +  x'_2 \sin(s/r)=0.$$
        On the other hand,
        \begin{equation*}
            \lambda=\langle \beta'\wedge w, w'\rangle=-\left(\sqrt{1-r^2}\right)[x'_1\cos(s/r)+x'_2\sin(s/r)]=0.
        \end{equation*}       
        Which is a contradiction, since $\lambda\neq0$. \\
        
        \noindent{\bf Case (III)}: $a=0$ and $y'_3\neq0$.   
        
       Considering $a=0$, from Lemma \ref{T3} we have \eqref{sist111 a=0}. Then, taking the first order derivative of the first equation in \eqref{sist111 a=0} and considering Lemma \ref{Triedo de Darboux} we get
        \begin{eqnarray*}
            0 &=& \dfrac{d}{ds}\langle L'(0)\beta, w'\wedge w\rangle\\
              &=& \langle (\sigma'(0)\Gamma(0)\beta+\Gamma'(0)\beta+\Theta'(0))', w'\wedge w\rangle + \langle L'(0)\beta, (w'\wedge w)'\rangle\\
              &=& \langle c\beta', w'\wedge w\rangle + \langle L'(0)\beta, Jw'\rangle\\
              &=& c\langle \beta', w'\wedge w\rangle + J\langle L'(0)\beta, w'\rangle.
        \end{eqnarray*}       
       Then, combining the above  equation with \eqref{complambda} we get
       \begin{eqnarray}\label{Lbetaw'}
            c\lambda &=& J\langle L'(0)\beta, w'\rangle.
       \end{eqnarray}
       Thus, from the second equation of \eqref{sist111 a=0} we have     
       \begin{eqnarray*}
            2c\lambda^2=-J^2.
       \end{eqnarray*}
        
        It is worth to notice that if $c=0$, then $J=0$, and from Theorem \ref{Sol Trivial}, $X$ is trivial. Therefore, from now on we consider $c<0$.
        
        Therefore, $w$ is a curve on $\mathbb{S}^2$ of constant geodesic curvature $J=\sqrt{-2c\lambda^{2}}$. Taking the first order derivative of the second equation in \eqref{sist111 a=0}, 
        \begin{eqnarray*}
            0 &=& \dfrac{d}{ds} \left(\langle L'(0)\beta, w'\rangle+\dfrac{1}{2\lambda}J\right)\\
              &=& \langle (L'(0)\beta)', w'\rangle+\langle L'(0)\beta, w''\rangle\\
              &=& \langle c\beta', w'\rangle + \langle L'(0)\beta, J w'\wedge w - w\rangle\\
              &=& c\langle \beta', w'\rangle + J\langle L'(0)\beta, w'\wedge w\rangle-\langle L'(0)\beta, w\rangle.
        \end{eqnarray*}         
        Then, from \eqref{regrada R} and the firs equation of system \eqref{sist111 a=0}, we have $$\langle L'(0)\beta, w\rangle=0.$$ Note that, since $\{w,w',w\wedge w'\}$ is a orthonormal frame in $\mathbb{R}^3$, from the above equation, the first equation in \eqref{sist111 a=0} and \eqref{Lbetaw'} we obtain
        \begin{eqnarray*}
            L'(0)\beta &=& \dfrac{c\lambda}{J}w'
            \end{eqnarray*}
            which implies that
            \begin{eqnarray*}
      c\beta+\Theta'(0) &=& \dfrac{c\lambda}{J}w'.
        \end{eqnarray*}
        From \eqref{bene1}, a straightforward computation gives us                  
        \begin{eqnarray*}
            \beta' &=& \dfrac{\lambda}{J}w''\\ 
                   &=& \dfrac{\lambda}{J}(J w'\wedge w - w)\\ 
                   &=& \lambda w'\wedge w - \dfrac{\lambda}{J}w. 
        \end{eqnarray*} 
        From the third equation of \eqref{sist111 a=0} we obtain 
        $$ 0=\langle\beta',\,w\rangle= \lambda \langle w'\wedge w,\,w\rangle - \dfrac{\lambda}{J}\langle w,\,w\rangle. $$
       Therefore, 
        $\frac{1}{\sqrt{-2c}}=0$ which is a contradiction.
    \end{proof}

    \begin{proof}[Proof of Theorem \ref{C40}]
        Let $X:U\subset \mathbb{R}^2\longrightarrow \mathbb{R}^{3}$ be a surface of revolution in $\mathbb{R}^{3}$ satisfying \eqref{regrada R}.
        Consider a one parameter family of surfaces ${X}^{t}(s,u)=L(t)X(s,u)$ where $L(t)$ is a homothetic helicoidal motion of $\mathbb{R}^3$.
        We can rewrite $\widehat{X}$ in the following form
        \begin{equation*}
            \widehat{X}^{t}(s,u)=\sigma(t)\Gamma(t)X(s,u)+\Theta(t), 
        \end{equation*}
        where $\Gamma$ and $\Theta$ stand for rotations and translations in $\mathbb{R}^{3}$, respectively. Here, the dilation function $\sigma$ determines the scaling and $\sigma(0)=1.$
        That being said, $\widehat{X}^{t}$ is a solution to the MCF in $\mathbb{R}^3$ if and only if
        $$\begin{cases}
            \dfrac{\partial\widehat{X}^{t}}{\partial t}=H^{t}(s,u)N^{t}(s,u);\\\\
            \widehat{X}^{0}(s,u)= X(s,u),\\
        \end{cases}$$  
        for all $t\in \mathfrak{I}\subset\mathbb{R}.$
            
        Considering that $\widehat{X}$ is a solution for the MCF, at $t=0$ we have
        \begin{equation}\label{soliton r340}
                \langle \sigma'(0)\Gamma(0)X+\Gamma'(0)X+\Theta'(0),N\rangle=H.
        \end{equation}  
                
        Let $\Gamma$ and $\Theta$ be rotation and  translation, respectively, of $\mathbb{R}^3$ given by
        $$\Gamma(t)=
        \begin{pmatrix}
            \cos{\xi(t)} & -\sin{\xi(t)} & 0\\ 
            \sin{\xi(t)} &  \cos{\xi(t)} & 0\\ 
                       0 &             0 & 1\\ 
        \end{pmatrix}\quad\mbox{and}\quad \Theta(t)=(0,\,0,\,\zeta(t))$$
        such that 
        \begin{eqnarray}\label{isometriatheo2240}
            \Gamma'(0)=a
            \begin{pmatrix}
                0 & -1 &  0\\ 
                1 &  0 &  0\\ 
                0 &  0 &  0\\ 
            \end{pmatrix},\quad\Theta'(0)=b(0,\,0,\,1)\quad\mbox{and}\quad\sigma'(0)=c,
        \end{eqnarray}
        where $a=\xi'(0)$ e $b=\zeta'(0)$. Here, we are assuming $\xi(0)=0$ (cf. \cite{Halldorsson1}). 
           
        At this point, we can go faster in the computation since the strategy is the same of the past theorems. Thus, a straightforward computation from     \eqref{Nrevolution}, \eqref{media r31}, \eqref{soliton r340} and \eqref{isometriatheo2240} gives us 
        \begin{eqnarray}\label{edorevolucao}
            \phi(\phi''\psi'-\phi'\psi'')-\psi'(\phi'^2 + \psi'^2)=2\phi(\phi'^2 + \psi'^2)[c(\psi'\phi-\psi\phi')-b\phi'].
        \end{eqnarray}
           
        Now, consider $\alpha(s)=(0,\,\phi(s),\,\psi(s))$ a plane curve of curvature $\kappa$ given by
        $$\kappa =\dfrac{-\psi''\phi'+\phi''\psi'}{(\phi'^2 + \psi'^2)^{3/2}}.$$
        Using the curvature of $\alpha$ in \eqref{edorevolucao} we obtain
        \begin{eqnarray*}
            \kappa=\frac{1}{|\tau|}\left(2 c\langle\alpha,\,\eta\rangle+ 2b\langle e_3,\,\eta\rangle+\dfrac{\langle e_3,\,\tau\rangle}{\phi}\right),
        \end{eqnarray*}
        where $\tau=\alpha'(s)$, $\eta=(0,\psi',\,-\phi')$ and $e_3=(0,\,0,\,1)$.
    \end{proof}    
    \begin{lemma}\label{C2}
        Let $X:U\subset \mathbb{R}^2\longrightarrow \mathbb{R}^{3}$ be a cylindrical ruled surface in $\mathbb{R}^{3}$ satisfying \eqref{regrada R}, i.e., $w'=0$. 
        Then $\widehat{X}^{t}(s,\,u)=L(t)X(s,\,u)$ is a self-similar solution to the MCF with initial condition $X$ if and only if
        \begin{equation}\label{sode2}\
            \begin{cases}
                \langle \Gamma'(0)w, \beta'\wedge w\rangle=0\\
                \langle L'(0)\beta,\, \beta'\wedge w\rangle=\dfrac{1}{2(|\beta'|^{2}-\langle\beta',\,w\rangle^{2})}\langle \beta'',\, \beta'\wedge w\rangle,
            \end{cases}
        \end{equation}
        where $L(t)=\sigma(t)\Gamma(t)+\Theta(t)$ stands for a homothety in $\mathbb{R}^{3}$, $\Gamma$ is rotation, $\Theta$ translation and $\sigma$ dilation.  
        \end{lemma}
    
        \begin{proof}[Proof of Lemma \ref{C2}]
            Let $X:U\subset \mathbb{R}^2\longrightarrow \mathbb{R}^{3}$ be a noncylindrical ruled surface in $\mathbb{R}^{3}$ satisfying \eqref{regrada R}.
            Consider a one parameter family of surfaces ${X}^{t}(s,u)=L(t)X(s,u)$ where $L(t)$ is a homothetic helicoidal motion of $\mathbb{R}^3$.
            We can rewrite $\widehat{X}$ in the following form
            \begin{equation*}
                \widehat{X}^{t}(s,u)=\sigma(t)\Gamma(t)X(s,u)+\Theta(t), 
            \end{equation*}
            where $\Gamma$ and $\Theta$ stand for rotations and translations in $\mathbb{R}^{3}$, respectively. Here, the dilation function $\sigma$ determines the scaling and $\sigma(0)=1.$
            That being said, $\widehat{X}^{t}$ is a solution to the MCF in $\mathbb{R}^3$ if and only if
            $$\begin{cases}
                \dfrac{\partial\widehat{X}^{t}}{\partial t}=H^{t}(s,u)N^{t}(s,u);\\\\
                \widehat{X}^{0}(s,u)= X(s,u),\\
            \end{cases}$$  
            for all $t\in \mathfrak{I}\subset\mathbb{R}.$    
    
            Considering that $\widehat{X}$ is a solution for the MCF, at $\tau=0$ we have
            \begin{equation}
                \langle \sigma'(0)X+\Gamma'(0)X+\Theta'(0),N\rangle=H.\nonumber
            \end{equation} 
            Moreover, the coefficients of the first and second fundamental formulas and the normal vector field are given by, respectively, 
            \begin{align*}
                E=|\beta'|^{2}; && F = \langle\beta',\,w\rangle; &&  G = 1.
            \end{align*}   
            \begin{align*}
                e=\dfrac{1}{\sqrt{|\beta'|^{2}-F^{2}}}\langle\beta'',\,\beta'\wedge w\rangle;&& f = g =0;\,&&  N=\dfrac{\beta'\wedge w}{\sqrt{|\beta'|^{2}-F^{2}}}.
            \end{align*} 
            Therefore, $$H=\frac{Eg-2fF+Ge}{2(EG-F^2)}=\dfrac{1}{2(|\beta'|^{2}-F^{2})^{3/2}}\langle\beta'',\,\beta'\wedge w\rangle.$$        On the other hand,
            \begin{eqnarray*}
                \langle\sigma'(0)X+\Gamma'(0)X+\Theta'(0),N\rangle=\dfrac{1}{\sqrt{|\beta'|^{2}-F^{2}}}\left(\langle L'(0)\beta,\, \beta'\wedge w\rangle+u\langle\Gamma'(0)w,\,\beta'\wedge w\rangle\right).
            \end{eqnarray*}
            Combining the two equations above we get the system of equations \eqref{sode2}.
        \end{proof}
    \begin{proof}[Proof of Theorem \ref{C3}]  
        Let $L=\sigma\Gamma+\Theta$ be an homothetic helicoidal motion of $\mathbb{R}^3$ given by
        $$\Gamma(t)=
        \begin{pmatrix}
            1 &            0 &            0 \\ 
            0 & \cos{\xi(t)} & -\sin{\xi(t)}\\ 
            0 & \sin{\xi(t)} &  \cos{\xi(t)}\\ 
        \end{pmatrix}\quad\mbox{and}\quad \Theta(t)=(\zeta(t),\,0,\,0)$$
        such that 
        \begin{eqnarray}\label{isometriatheo22}
            \Gamma'(0)=a
            \begin{pmatrix}
                0 & 0 &  0\\ 
                0 &  0 &  -1\\ 
                0 & 1 &  0\\ 
            \end{pmatrix},\quad\Theta'(0)=b(1,\,0,\,0)\quad\mbox{and}\quad\sigma'(0)=c,
        \end{eqnarray}
        where $a=\xi'(0)$, $b=\zeta'(0)$ and $\sigma'(0)=c$. Here, we are assuming $\xi(0)=0$ (cf. \cite{Halldorsson1}). 
                
        \noindent{\bf Item (I)} Suppose the family $\widehat{X}^{t}(s,u)=L(t)X(s,u)$ is a solution to the MFC on $\mathbb{R}^3$ with initial condition $X(s,u)=\beta(s)+u w(s)$, where $\beta(s)=(0,\,h(s),\,q(s))$ and $w(s)=(x_0,\,y_0,\,z_0).$ Thus,  $\beta'=(0,\,h',\,q')$, $\beta'\wedge w = (z_{0}h'-y_{0}q',\,x_{0}q',\,-x_{0}h')$ and $L'(0)\beta=(b,\,ch-aq,\,cq+ah).$

        Hence, from the first equation of system \eqref{sode2} we have
        \begin{eqnarray}\label{edo23}
            \langle \Gamma'(0)w,\, \beta'\wedge w\rangle=-ax_{0}(z_{0}q'+y_{0}h')=0. 
        \end{eqnarray}  
        Therefore, from the second equation of system \eqref{sode2} we have
        \begin{eqnarray}\label{edo13}
            2cx_0 (hq'-qh')+2b(z_{0}h'-y_{0}q')-2ax_{0}(qq'+hh')\nonumber\\
            =\dfrac{x_{0}(q'h''-q''h')}{(h')^{2}+(q')^{2}-(y_{0}h'+z_0q')^{2}}.
        \end{eqnarray}
           
        \noindent{\bf Item (II)} Now, let us consider the case where $\beta(s)=(h(s),\,q(s),\,0)$ and $w(s)=(x_0,\,y_0,\,z_0).$ Thus,  $\beta'=(h',\,q',\,0)$, $\beta'\wedge w = (z_{0}q',\,-z_{0}h',\,y_0h'-x_{0}q')$ and $L'(0)\beta=(ch+b,\,cq,\,aq).$
        Since $|w|=1$, from the first equation of system \eqref{sode2} we have
        \begin{eqnarray}\label{eqtheo31}
            \langle \Gamma'(0)w,\, \beta'\wedge w\rangle=a[(z_{0}^{2}+y_{0}^{2})h'-x_{0}y_{0}q']=a[(1-x_{0}^{2})h'-x_{0}y_{0}q']=0. 
        \end{eqnarray}  
        Therefore, from the second equation of system \eqref{sode2} we have
        \begin{eqnarray}\label{eqtheo32}
            2cz_0 (hq'-qh')+2bz_{0}q'+2aq(y_0h'-x_0q')\nonumber\\
            =\dfrac{z_{0}(q'h''-q''h')}{(h')^{2}+(q')^{2}-(x_{0}h'+y_0q')^{2}}.
        \end{eqnarray}
            
        Now, considering $\tau$, $\eta$ and $\kappa$ are the tangent vector, the normal vector and the curvature of $\beta$, respectively, we conclude our proof.
    \end{proof}
    \section{Self-similar solutions}\label{examples}   
        \subsection{MCF by surfaces of revolution} 
            Let $\alpha(s)=(0,\,\phi(s),\,\psi(s))$ be a plane curve parameterized by arc length satisfying Theorem \ref{C40}.
            Thus, combining $\phi'^2+\psi'^2=1$ with \eqref{caracteriSR} we get  
            \begin{eqnarray}\label{ccc1}
                \phi''=\psi'\left(2c(\phi\psi'-\phi'\psi) - 2b\phi'+\frac{\psi'}{\phi}\right);
            \end{eqnarray}
            \begin{eqnarray}\label{ccc2}
                \psi''=-\phi'\left(2c(\phi\psi'-\phi'\psi) - 2b\phi'+\frac{\psi'}{\phi}\right).
            \end{eqnarray}   
            It is important to highlight that $c$ and $b$ are constants related to dilations and translations of surface of revolution generated by $\alpha$, respectively. Therefore, when we have $c=0$ we do not have dilation and when $b=0$ we do not have translation. Moreover, if $c$ is negative the curve shrinks and for $c$ positive the curve expands over time.
     
            Let us exhibit some numerical solutions for \eqref{ccc1} and \eqref{ccc2}. Not always is possible to solve systems like this one. So, we can only see the solutions through a qualitative analysis. This approach is useful to understand the geometry of the solutions (cf. \cite{Halldorsson1,Halldorsson,Hungerbuhler,dosReis}). Here, we will only plot some solutions. Moreover, we will see that some of them behave like some solutions to the CSF in $\mathbb{R}^2$ given by \cite{Halldorsson1} while others don't.
            
            Let us describe the behavior of solutions for \eqref{ccc1} and \eqref{ccc2}:            
            
            \begin{figure}[ht]
                \begin{minipage}[b]{0.45\linewidth}
                	\centering
                	\includegraphics[scale=0.3]{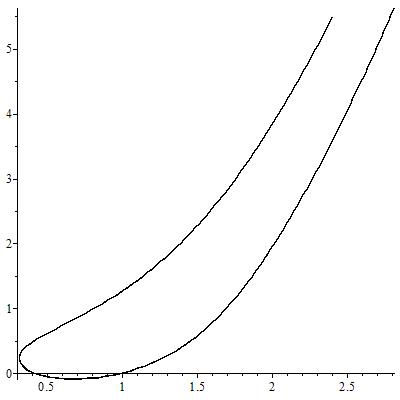}\\
                	\caption{$b=1;\,c=0$.\\}
                	\label{f1}
                \end{minipage} \hfill
                \begin{minipage}[b]{0.45\linewidth}
                	\centering
                	\includegraphics[scale=0.3]{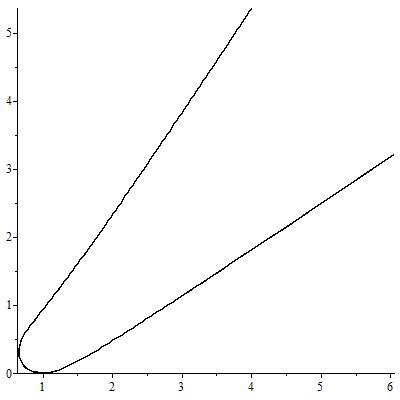}
                	\caption{$b=1;\,c=1$.}
                	\label{f2}
            	\end{minipage}
            \end{figure}
            {\bf Figure \ref{f1}:} With initial conditions $\phi(0)=1,\,\psi(0)=0,\,\phi'(0)=\sqrt{0,75},\,\psi'(0)=\sqrt{0,25}$, we have a curve translating in the Euclidean plane $\mathbb{R}^{2}$.  
            
            {\bf Figure \ref{f2}:} With initial conditions $\phi(0)=1,\,\psi(0)=0,\,\phi'(0)=1,\,\psi'(0)=0$, we have a curve translating and expanding in the Euclidean plane $\mathbb{R}^{2}$.
            \newpage
            \begin{figure}[h]
                \begin{minipage}[b]{0.45\linewidth}\centering
            	    \includegraphics[scale=0.3]{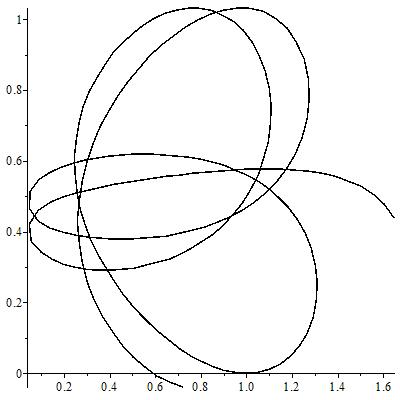}
            	    \caption{$b=1;\,c=-2$.}
            	    \label{f3}
            	\end{minipage} \hfill
            	\begin{minipage}[b]{0.45\linewidth}
            		\includegraphics[scale=0.3]{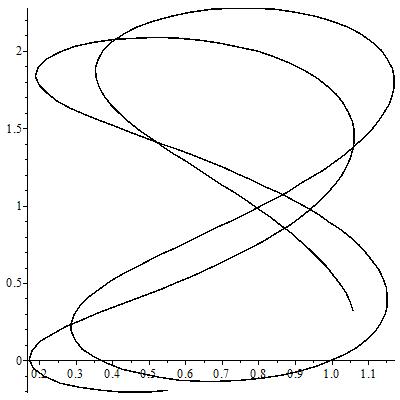}\centering
            	    \caption{$b=1;\,c=-1$.}
            	    \label{f4}
            	\end{minipage}  
            \end{figure}
            {\bf Figure \ref{f3}:} With initial conditions $\phi(0)=1,\,\psi(0)=0,\,\phi'(0)=1,\,\psi'(0)=0$, we have a curve translating and shrinking in the Euclidean plane $\mathbb{R}^{2}$.
                
            {\bf Figure \ref{f4}:} With initial conditions $\phi(0)=1,\,\psi(0)=0,\,\phi'(0)=1,\,\psi'(0)=1$, we have a curve translating and shrinking in the Euclidean plane $\mathbb{R}^{2}$.            
            
            \begin{figure}[ht]
            	\begin{minipage}[b]{0.45\linewidth}
            		\includegraphics[scale=0.3]{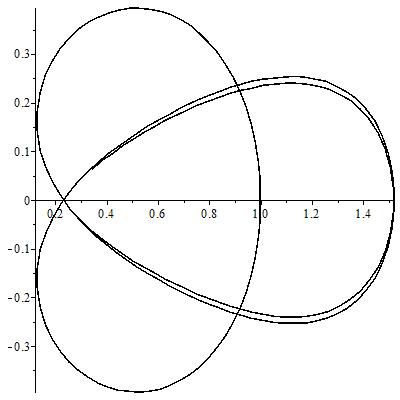}\centering
            	    \caption{$b=0 ;\,c=-2$}
            	    \label{f6}
            	\end{minipage} \hfill
            	\begin{minipage}[b]{0.45\linewidth}
            		\includegraphics[scale=0.3]{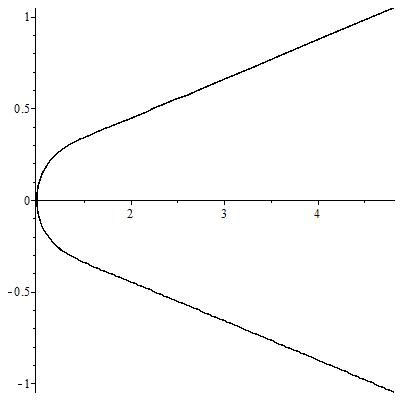}\centering
            	    \caption{$b=0 ;\,c=2$}
            	    \label{f5}
            	\end{minipage}
            \end{figure}       
            {\bf Figure \ref{f6}:} With initial conditions $\phi(0)=1,\,\psi(0)=0,\,\phi'(0)=0,\,\psi'(0)=1$, we have a curve expanding in the Euclidean plane $\mathbb{R}^{2}$.  
            
            {\bf Figure \ref{f5}:} With initial conditions $\phi(0)=1,\,\psi(0)=0,\,\phi'(0)=0,\,\psi'(0)=1$, we have a curve shrinking in the Euclidean plane $\mathbb{R}^{2}$.                
            
            \begin{figure}[ht]
            	\begin{minipage}[b]{0.45\linewidth}
            		\includegraphics[scale=0.3]{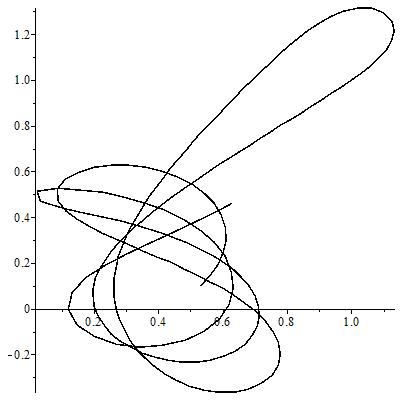}\centering
            	    \caption{$b=1 ;\,c=-4$}
            	    \label{f7}
            	\end{minipage} \hfill
            	\begin{minipage}[b]{0.45\linewidth}
            		\includegraphics[scale=0.3]{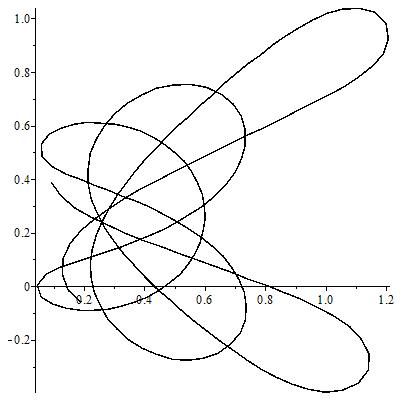}\centering
            	    \caption{$b=1 ;\,c=-4$}
            	    \label{f8}
            	\end{minipage} 
            \end{figure}   
            {\bf Figure \ref{f7}:} With initial conditions $\phi(0)=1,\,\psi(0)=1,\,\phi'(0)=1,\,\psi'(0)=1$, we have a curve translating and shrinking in the Euclidean plane $\mathbb{R}^{2}$.
                
            {\bf Figure \ref{f8}:} With initial conditions $\phi(0)=0,5,\,\psi(0)=0,5,\,\phi'(0)=-1,\,\psi'(0)=1$, we have a curve translating and shrinking in the Euclidean plane $\mathbb{R}^{2}$.            
            
            \begin{figure}[ht]	
                \begin{minipage}[b]{0.45\linewidth}
            	    \includegraphics[scale=0.3]{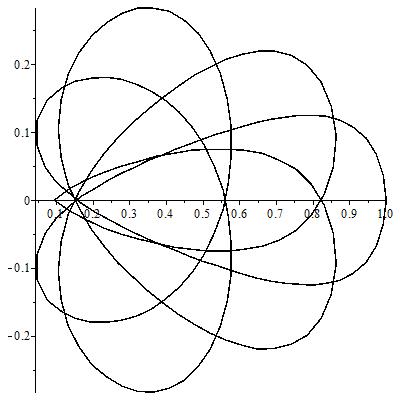}\centering
            	    \caption{$b=0 ;\,c=-6$}
            	    \label{f9}
                \end{minipage} \hfill
            	\begin{minipage}[b]{0.45\linewidth}
            		\includegraphics[scale=0.3]{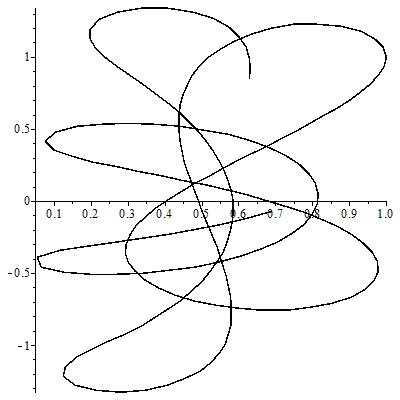}\centering
            	    \caption{$b=0 ;\,c=-2$}
            	    \label{f10}
            	\end{minipage} \hfill
            \end{figure}     
            {\bf Figure \ref{f9}:} With initial conditions $\phi(0)=1,\,\psi(0)=0,\,\phi'(0)=0,\,\psi'(0)=1$, we have a curve shrinking in the Euclidean plane $\mathbb{R}^{2}$.
            
            {\bf Figure \ref{f10}:} With initial conditions $\phi(0)=1,\,\psi(0)=1,\,\phi'(0)=0,\,\psi'(0)=1$, we have a curve shrinking in the Euclidean plane $\mathbb{R}^{2}$.            
 
        \subsection{MCF by cylindrical surfaces} 
            \begin{example}[Trivial solutions]
                The following are trivial cases from Theorem \ref{C3}. 
                
                Let us start with Item (I).
                \begin{itemize}
                    \item For $x_0=0$ and $b=0$, Item (I) is trivial.\  
                    \item For $x_0\neq0$ and $y_0=0$ (or $z_0=0$), from \eqref{edo23} we have $q'=0$ (or $h'=0$). Thus, from \eqref{edo13} we get $(bz_0-cq-ax_0 h)h'=0.$
                    since $q$ is constant, the only possibility is $h$ be constant. Therefore, $X(s,\,u)=(ux_0,\,h_0,\,q_0+uz_0)$ is a trivial solution for the Item (I).
                \end{itemize}

                In what follows, we focus on Item (II). 
                \begin{itemize}  
                    \item For $z_0=0$ and $a=0$, Item (II) is trivial.\
                    \item For $x_0=0$ and $h(s)=-\dfrac{b}{c}$, Item (II) is trivial.\
                 
                \end{itemize}
            \end{example}
            \begin{example}\label{grim}[Exact translating solutions for the MCF by cylindrical surfaces]  
                Now, we will consider Theorem \ref{C3}, Item (I). Moreover, we will assume $a=c=0$, i.e., just translating solutions.
   
                The qualitative analysis for self-similar solutions was consider by several authors. Thus, the aim of this example is to provide some exact solutions for Theorem \ref{C3}. As expected, the grim reaper solution is one of these exact solutions. 
   
                Thus, for \eqref{edo13} with $\beta(s)=(0,\,s,\,q(s))$ we get 
                $$2b(z_0 - y_0 q')=\dfrac{-x_0 q''}{1+q'^{2}-(y_0 + z_0 q')^{2}}.$$
   
                We provide now the first family of soliton solutions considering $b=1/2$, $x_0\neq0$, $y_0\neq0$ and $z_0=0$.  Then we have
                \begin{eqnarray}\label{t1}
                    y_0 q'=\dfrac{x_0 q''}{1+q'^{2}-y_0^{2}}.
                \end{eqnarray}
                Assuming, for instance, $y_0=1/2$ and $x_0=-1$, $$q(s)=\pm\dfrac{4\sqrt{3}}{3}\arctan(\frac{1}{2}\sqrt{e^{3s/2}-4}).$$
                Therefore,
                \begin{eqnarray}\label{sol1}
                    X(s,\,u)=(-u,\,s+\frac{u}{2},\,\dfrac{4\sqrt{3}}{3}\arctan(\frac{1}{2}\sqrt{e^{3s/4}-4})).
                \end{eqnarray}
                \begin{figure}[!htb]
                    \begin{minipage}[b]{0.45\linewidth}
	                    \centering
                	    \includegraphics[scale=0.4]{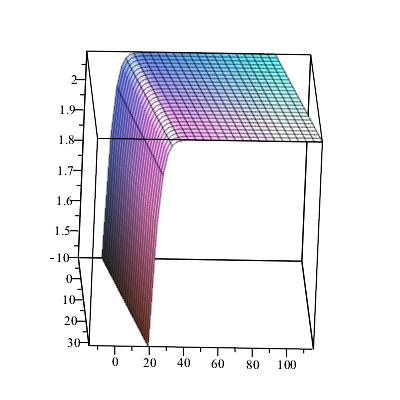}
                		\caption{Graph of \eqref{sol1}}
                	\end{minipage} \hfill
                	\begin{minipage}[b]{0.45\linewidth}
                	    \centering
                		\includegraphics[scale=0.4]{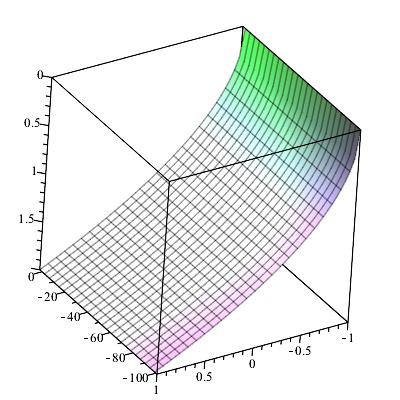}
                		\caption{Graph of \eqref{sol2}}
                	\end{minipage} \hfill
                \end{figure}
                
                Another case to consider is $x_0=-1$ and $y_0=1$. Thus, from \eqref{t1} we have $$q(s)=\pm\sqrt{2s+2}.$$
                Therefore, 
                \begin{eqnarray}\label{sol2}
                    X(s,\,u)=(-u,\,u+s,\,\sqrt{2s+2}).
                \end{eqnarray}

                The grim reaper solution is also an example. Begin by considering $x_0\neq0$, $z_0\neq0$ and $y_0=0$. Thus,
                $$2bz_0=\dfrac{-x_0 q''}{1+q'^{2}-( z_0 q')^{2}}.$$ For instance, let $b=x_0=1$ and $z_0=1/2$. Then we have $$-q''=1+\frac{3}{4}(q')^{2}.$$
                Thus, $$q(s)=\frac{2}{3}\log\left(\frac{3}{4}(\sin(\frac{\sqrt{3}}{2}s)-\cos(\frac{\sqrt{3}}{2}s))^{2}\right)$$ is a solution for the above ODE.
                Therefore, $$X(s,\,u)=\left(u,\,s,\,\frac{u}{2}+\frac{2}{3}\log\left(\frac{3}{4}(\sin(\frac{\sqrt{3}}{2}s)-\cos(\frac{\sqrt{3}}{2}s))^{2}\right)\right)$$ is the grim reaper solution.
                \begin{figure}[ht]
	                \centering
	                \includegraphics[scale=0.4]{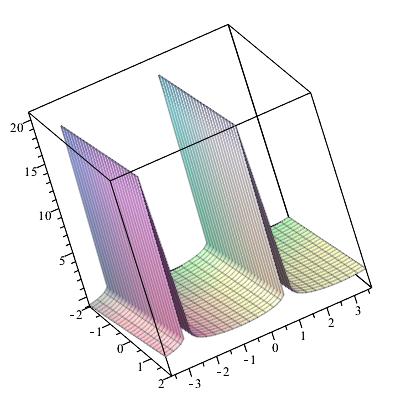}
	                \caption{Grim reaper}
                \end{figure}
            \end{example}
    \newpage
    \begin{acknowledgement}
        This work was done while the third author was a postdoc at Instituto de Matemática e Estatística, Universidade Federal de Goiás, Brazil. He is grateful to the hosted institution for the scientific atmosphere that it has provided during his visit.
    \end{acknowledgement}
	
\end{document}